\documentclass[12pt]{amsart}

\usepackage[margin=1.15in]{geometry}
\usepackage{amscd,amssymb, amsmath,amsfonts, wasysym, mathrsfs, mathtools,enumerate,hhline,color}
\usepackage{graphicx}
\usepackage[all, cmtip]{xy}

\usepackage{url}

\definecolor{hot}{RGB}{65,105,225}

\usepackage[pagebackref=true,colorlinks=true, linkcolor=hot ,  citecolor=hot, urlcolor=hot]{hyperref}

\theoremstyle{plain}
\newtheorem{theorem}{Theorem}[section]
\newtheorem{prop}[theorem]{Proposition}

\newtheorem{cor}[theorem]{Corollary}

\theoremstyle{definition}

\newtheorem*{ex*}{Example}

\newtheorem*{question}{Question}

\newcommand\sO{{\mathcal O}}

\newcommand\sH{{\mathcal H}}

\newcommand\pp{{\mathbb{P}}}

\newcommand\zz{{\mathbb{Z}}}

\newcommand\cc{{\mathbb{C}}}

\DeclareMathOperator{\MLdeg}{MLdeg}

\title[Fermat hypersurfaces]{Maximum likelihood degree of Fermat hypersurfaces via Euler characteristics}
\begin{document}

\author{Botong Wang}
\email{botong.wang@wis.kuleuven.be }
\address{KU Leuven, Department of Mathematics,
Celestijnenlaan 200B, B-3001 Leuven, Belgium} 


\begin{abstract}
Maximum likelihood degree of a projective variety is the number of critical points of a general likelihood function. In this note, we compute the Maximum likelihood degree of Fermat hypersurfaces. We give a formula of the Maximum likelihood degree in terms of the constants $\beta_{\mu, \nu}$, which is defined to be the number of complex solutions to the system of equations $z_1^\nu=z_2^\nu=\cdots=z_\mu^\nu=1$ and $z_1+\cdots +z_\mu+1=0$. 
\end{abstract}


\maketitle
\section{introduction}
The maximum likelihood estimate is a fundamental problem in statistics. Maximum likelihood degree is the number of potential solutions to the maximum likelihood estimation problem on a projective variety. When the variety is smooth, Huh \cite{Hu} showed that the Maximum likelihood degree is indeed a topological invariant. If the variety is a general complete intersection, the maximum likelihood degree is computed in \cite{CHKS} (see also \cite{HS}). 

In a recent preprint \cite{AAGL}, Agostini, Alberelli, Grande and Lella studied the maximum likelihood degree of Fermat hypersurfaces. They obtained formulas for the maximum likelihood degree of a few special families of Fermat surfaces. However, their approach is through a case-by-case study. 

In this note, we propose to compute the Maximum likelihood degree of Fermat hypersurfaces in a more systematic way via topological method. In general, the formula given in \cite{CHKS} does not work for all the Fermat hypersurfaces, because the intersection of hypersurfaces 
$$\{x_0^d+x_1^d+\cdots+x_n^d=0\}\cap \{x_0+x_1+\cdots+x_n=0\}\subset \pp^n$$
may not be transverse. We will compute the error terms introduced by the non-transverse intersections. The main ingredient is Milnor's result on the topology of isolated hypersurfaces singularities. This topological approach is closely related to the approach of \cite{BW} and \cite{RW}. In fact, for an isolated hypersurface singularity, the Euler obstruction is up to a sign equal to the Milnor number plus one. So we essentially apply the ideas of \cite{BW} and \cite{RW} to these particular examples. 

First, let us recall the definition of Maximum likelihood degree. Let $\pp^n$ be the $n$-dimensional complex projective space with homogeneous coordinates $(x_0, x_1, \ldots, x_n)$. Denote the coordinate plane $\{x_i=0\}\subset \pp^n$ by $H_i$, and the hyperplane $\{x_0+x_1+\cdots+x_n=0\}$ by $H_+$. Let the index set $\Lambda=\{0, 1, \ldots, n, +\}$, and let $\sH=\bigcup_{\lambda\in \Lambda}H_\lambda$. Let $X\subset\pp^n$ be a complex projective variety. Denote the smooth locus of $X$ by $X_{\textrm{reg}}$. The \textbf{Maximum likelihood degree} of $X$ is defined to be the number of critical points of the likelihood function 
$$l_u=\frac{x_0^{u_0}x_1^{u_1}\cdots x_n^{u_n}}{(x_0+x_1+\cdots+x_n)^{u_0+u_1+\cdots+u_n}}$$
on $X_{\textrm{reg}}\setminus \sH$ for generic $(u_i)_{0\leq i\leq n}\in \zz^{n+1}$. 
\begin{theorem}\label{thm}
Denote the Fermat hypersurface $\{x_0^d+x_1^d+\cdots+x_n^d=0\}\subset \pp^n$ by $F_{n, d}$, and denote its maximum likelihood degree by $\MLdeg(F_{n, d})$. Then,
\begin{equation}\label{main}
\MLdeg(F_{n, d})=d+d^2+\cdots+d^n-\sum_{0\leq j\leq n-1}{n+1\choose j}\beta_{n-j, d-1}
\end{equation}
where $\beta_{\mu, \nu}$ is the number of complex solutions of the system of equations
\begin{gather*}
z_1^{\nu}=z_2^{\nu}=\ldots=z_{\mu}^{\nu}=1\\
z_1+\ldots+z_{\mu}+1=0. 
\end{gather*}

\end{theorem}
When $\mu$ or $\nu$ is small, $\beta_{\mu,\nu}$ can be easily calculated. For example, 
\begin{equation}
\beta_{\mu, 1}=0.
\end{equation}

\begin{equation}
\beta_{1, \nu}=
\begin{cases}
0 & \text{if $\nu$ is odd},\\
1 & \text{if $\nu$ is even}.
\end{cases}
\end{equation}

\begin{equation}
\beta_{2, \nu}=
\begin{cases}
2 & \text{if $\nu$ is divisible by 3},\\
0 & \text{otherwise}.
\end{cases}
\end{equation}

With these calculations, we recover all the closed formulas in \cite{AAGL}. 
\begin{cor}
\begin{equation}
\MLdeg(F_{n, 2})=2^{n+1}-2
\end{equation}

\begin{equation}
\MLdeg(F_{2, d})=
\begin{cases}
d^2+d & \text{if $d\equiv 0, 2\mod 6$},\\
d^2+d-3 & \text{if $d\equiv 3, 5\mod 6$},\\
d^2+d-2 & \text{if $d\equiv 4\mod 6$},\\
d^2+d-5 & \text{if $d\equiv 1\mod 6$}.
\end{cases}
\end{equation}
\end{cor}

When $\nu$ is a power of a prime number, we have formulas to compute $\beta_{\mu, \nu}$. Equivalently, when $d-1$ is a power of a prime number, we have closed formulas for $\MLdeg(F_{n, d})$. In fact, by a straight forward computation one can deduce the following corollary from Theorem \ref{thm} and Proposition \ref{com}. 
\begin{cor}
Suppose $d-1=p^r$, where $p$ is a prime number and $r$ is a positive integer. Then 
$$
\MLdeg(F_{n, d})=d+d^2+\cdots+d^n-\frac{1}{d-1}\sum \frac{(n+1)!}{\left(n+1-p(s_1+\cdots+s_k)\right)!\cdot \left((s_1)!\cdots(s_k)!\right)^p}
$$
where $k=\frac{d-1}{p}$ and the sum is over all nonnegative integers $s_1, \ldots, s_k$ with $1\leq s_1+\cdots+s_k\leq \frac{n+1}{p}$. 
\end{cor}

To find a general formula for $\beta_{\mu, \nu}$ would be a very hard question in number theory and combinatorics. In fact, determining when $\beta_{\mu, \nu}\neq 0$ had been an open question for a long time, and it was solved by Lam and Leung \cite{LL} in 2000. 

Since the Fermat hypersurface $F_{n, d}$ is smooth, by \cite{Hu} $\MLdeg(F_{n, d})$ is equal to the the signed Euler characteristic $\chi(F_{n, d}\setminus \sH)$. In section \ref{two}, we will compute $\chi(F_{n, d}\setminus \sH)$, and we will postpone the technical calculation of the Milnor numbers to section \ref{three}. In the last section, we will briefly discuss what we know about the constants $\beta_{n, d}$. 

\subsection*{Acknowledgement}
We thank Jiu-Kang Yu and Zhengpeng Wu for helpful discussions about the constants $\beta_{\mu, \nu}$. 

\section{Computing the Euler characteristics}\label{two}
By the following theorem of Huh \cite{Hu}, we reduce the problem of computing $\MLdeg(F_{n, d})$ to computing $\chi(F_{n, d}\setminus \sH)$. Recall that in $\pp^n$, $\sH=\bigcup_{\lambda\in \Lambda}H_\lambda$ is the union of all coordinate hyperplanes and the hyperplane $H_{+}=\{x_0+x_1+\cdots+x_n=0\}$. 
\begin{theorem}[Huh, \cite{Hu}]\label{Huh}
If $X\subset \pp^n$ is a subvariety such that $X\setminus \sH$ is smooth, then
$$
\MLdeg(X)=(-1)^{\dim(X)}\chi(X\setminus \sH). 
$$
\end{theorem}

Since the Euler characteristic is additive for algebraic varieties, by the inclusion-exclusion principle,
\begin{equation}\label{inc}
\chi(X\setminus\sH)=\sum_{0\leq i\leq n}\;\sum_{\substack{\Lambda'\subset \Lambda\\ \left\vert \Lambda'\right\vert=i}}(-1)^i\chi(X\cap H_{\Lambda'})
\end{equation}
where $H_{\Lambda'}=\bigcap_{\lambda\in \Lambda'}H_\lambda$. 

The Fermat hypersurface $F_{n, d}=\{x_0^d+x_1^d+\cdots+x_n^d=0\}$ is invariant under any permutation of the coordinates. Therefore, (\ref{inc}) can be written as
\begin{equation}\label{total}
\chi(F_{n, d}\setminus \sH)=\sum_{0\leq i\leq n}(-1)^i\left({n+1 \choose i}\chi(F_{n, d}\cap V^i)+{n+1 \choose i-1}\chi(F_{n, d}\cap W^i)\right)
\end{equation}
where $V^i=\bigcap_{0\leq j\leq i-1} H_j$ and $W^i=H_+\cap \bigcap_{0\leq j\leq i-2}H_j$ ($W^0=\emptyset$ and $W^1=H_+$). 

$F_{n, d}\cap V^i$ is a smooth hypersurface in $\pp^{n-i}$ of degree $d$. Euler characteristics of such hypersurfaces only depend on $n-i$ and $d$, and they are calculated in \cite{D} Chapter 5, (3.7). However, it turns out that we don't have to compute each of these Euler characteristics. For now, we simply denote the Euler characteristic of a smooth degree $d$ hypersurfaces in $\pp^m$ by $e_{m, d}$. In particular,
\begin{equation}\label{smooth}
\chi(F_{n,d}\cap V^i)=e_{n-i, d}.
\end{equation}

$F_{n, d}\cap W^i$ is a possibly singular hypersurface in $W^i$ for $1\leq i\leq n$. In fact, $F_{n, d}\cap W^i$ is isomorphic to the intersection of the Fermat hypersurface $F_{n-i+1, d}\subset \pp^{n-i+1}$ and the hyperplane $\{x_0+x_1+\cdots+x_{n-i+1}=0\}$. Using Lagrange multiplier method, one can easily see that all the singular points of $F_{n, d}\cap W^i$ are isolated and there are exactly $\beta_{n-i+1, d-1}$ many of them. The Euler characteristics of such hypersurfaces can be computed using Milnor numbers. 
\begin{theorem}{\cite[Chapter 5 (4.4)]{D}}
For any singular point $P$ of $F_{n, d}\cap W^i$ we can define the Milnor number $\mu(F_{n, d}\cap W^i, P)$ by considering $F_{n, d}\cap W^i$ as a hypersurface of $W^i$. Then, 
\begin{equation}\label{sing}
\chi(F_{n, d}\cap W^i)=e_{n-i, d}+ (-1)^{n-i}\sum_{P}\mu(F_{n, d}\cap W^i, P)
\end{equation}
where the sum is over all the singular points $P$ of $F_{n, d}\cap W^i$.
\end{theorem}

\begin{prop}\label{num}
For any singular point $P$ of $F_{n, d}\cap W^i$, 
\begin{equation}\label{milnor}
\mu(F_{n, d}\cap W^i, P)=1.
\end{equation}
\end{prop}
We will postpone the proof of the proposition to next section. The next corollary follows immediately from (\ref{sing}) and (\ref{milnor}). 
\begin{cor}
\begin{equation}\label{sss}
\chi(F_{n, d}\cap W^i)=e_{n-i, d}+(-1)^{n-i}\beta_{n-i+1, d-1}. 
\end{equation}
\end{cor}
Now, combining (\ref{total}), (\ref{smooth}) and (\ref{sss}), we have
\begin{equation}\label{total2}
\chi(F_{n, d}\setminus \sH)=\sum_{0\leq i\leq n}(-1)^i\left({n+1 \choose i}e_{n-i, d}+{n+1 \choose i-1}\left(e_{n-i, d}+(-1)^{n-i}\beta_{n-i+1, d-1}\right)\right)
\end{equation}
Since 
$
{n+1\choose i}+{n+1\choose i-1}={n+2\choose i},
$
(\ref{total2}) is equivalent to
\begin{equation}\label{close}
\chi(F_{n, d}\setminus \sH)=\sum_{0\leq i\leq n}(-1)^i{n+2 \choose i}e_{n-i, d}+\sum_{1\leq i\leq n}(-1)^{n}{n+1 \choose i-1}\beta_{n-i+1, d-1}.
\end{equation}
Suppose $X$ is a general hypersurface of degree $d$ in $\pp^n$. Then (\ref{inc}) implies that
\begin{equation}\label{irene}
\chi(X\setminus \sH)=\sum_{0\leq i\leq n}(-1)^{i}{n+2\choose i}e_{n-i, d}
\end{equation}
The maximum likelihood degree of a general hypersurfaces is well-understood. 
\begin{prop}{\cite[1.11]{HS}}
The maximum likelihood of a general degree $d$ hypersurfaces in $\pp^n$ is equal to $d+d^2+\cdots+d^n$. 
\end{prop}
Combining the proposition, (\ref{irene}) and Theorem \ref{Huh}, we have 
\begin{equation}
\begin{split}
\sum_{0\leq i\leq n}(-1)^{i}{n+2\choose i}e_{n-i, d}&=\chi(X\setminus \sH)\\
&=(-1)^{n-1}\MLdeg(X)\\
&=(-1)^{n-1}(d+d^2+\cdots+d^n)
\end{split}
\end{equation}
Therefore, (\ref{close}) is equivalent to
\begin{equation}
\chi(F_{n, d}\setminus \sH)=(-1)^{n-1}(d+d^2+\cdots+d^n)+\sum_{1\leq i\leq n}(-1)^n{n+1\choose i-1}\beta_{n-i+1, d-1}
\end{equation}
Again, by Theorem \ref{Huh}, we have 
\begin{equation}
\MLdeg(F_{n, d})=d+d^2+\cdots+d^n-\sum_{1\leq i\leq n}{n+1\choose i-1}\beta_{n-i+1, d-1}
\end{equation}
which is the statement of Theorem \ref{thm}. 

\section{The Milnor numbers}\label{three}
We prove Proposition \ref{num} in this section. 

For the geometric meaning of Milnor number, we refer to \cite[Chapter 3]{D}. Here we compute the Milnor numbers using Jacobian ideals. Denote the ring of germs of holomorphic functions at $0\in \cc^l$ by $\sO$. Let $f\in \sO$ be a nonzero germ of holomorphic function such that the germ of hypersurface $f^{-1}(0)$ has an isolated singularity at the origin $0\in \cc^l$. The Jacobian ideal of $f$, denoted by $J_f$ is defined by
$$J_f=(\frac{\partial f}{\partial z_1}, \cdots, \frac{\partial f}{\partial z_l})\subset \sO$$
where $z_1, \ldots, z_l$ are the coordinates of $\cc^n$. 
\begin{theorem}{\cite[Chapter 3, (2.7)]{D}}\label{1st}
The Milnor number of $f^{-1}(0)$ at the origin, denoted by $\mu(f^{-1}(0), 0)$, is given by the formula
\begin{equation}
\mu(f^{-1}(0), 0)=\dim_\cc \sO/J_f. 
\end{equation}
\end{theorem}
Recall that $W^i=\{x_0=x_1=\cdots=x_{i-2}=x_0+x_1+\cdots+x_n=0\}\subset \pp^n$. Denote $y_j=x_{i-1+j}$, $0\leq j\leq n-i+1$. Then the intersection $F_{n, d}\cap W^i$ is isomorphic to the intersection
$$\{y_0^d+y_1^d+\cdots+y_{n-i+1}^d=0\}\cap \{y_0+y_1+\cdots+y_{n-i+1}=0\}$$
in $\pp^{n-i+1}$. Without lost of generality, we can work on the affine space $y_0\neq 0$, and rewrite the intersection in affine coordinates
$$\{1+\bar{y}_1^d+\cdots+\bar{y}_{n-i+1}^d=0\}\cap \{1+\bar{y}_1+\cdots+\bar{y}_{n-i+1}=0\}.$$
Here we use $\bar{y}_j$ to denote the corresponding affine coordinate of $y_j$, that is, $\bar{y}_j=y_j/y_0$. Suppose $(\xi_1, \ldots, \xi_{n-i+1})$ is a singular point of the above intersection. Then by Lagrange multiplier method,
\begin{equation}
\xi_1^{d-1}=\xi_2^{d-1}=\cdots=\xi_{n-i+1}^{d-1}=1.
\end{equation}

We can eliminate $\bar{y}_{n-i+1}$ by $\bar{y}_{n-i+1}=1-\bar{y_1}-\cdots-\bar{y}_{n-i}$. On this affine chart, $F_{n, d}\cap W^i$ is isomorphic to the hypersurface $\{f=0\}$ in $\cc^{n-i}$, where
\begin{equation}
f=1+\bar{y}_1^d+\cdots+\bar{y}_{n-i}^d+(1-\bar{y}_1-\cdots-\bar{y}_{n-i})^d.
\end{equation}
Let $z_j=\bar{y}_j-\xi_j$. Then 
\begin{equation}
f=1+(z_1+\xi_1)^d+\cdots+(z_{n-i}+\xi_{n-i})^d+(\xi_{n-i+1}-z_1-\cdots-z_{n-i})^d.
\end{equation}

\begin{prop}\label{2nd}
In the local ring $\sO$, the Jacobian ideal $J_f=(\frac{\partial f}{\partial z_1}, \ldots, \frac{\partial f}{\partial z_{n-i}})$ is equal to the maximal ideal $(z_1, z_2, \ldots, z_{n-i})$. 
\end{prop}
\begin{proof}
Notice that $\xi_j^{d-1}=1$ for all $1\leq j\leq n-i+1$. Therefore, 
\begin{equation*}
\begin{split}
\frac{\partial f}{\partial z_j}&=\frac{d(d-1)}{2}\cdot \xi_j^{d-2}z_j+\frac{d(d-1)}{2}\cdot \xi_j^{d-2} (z_1+\cdots+z_{n-i})+\text{higher degree terms}\\
&=\frac{d(d-1)}{2}\cdot \xi_j^{d-2}(z_1+\cdots+z_{j-1}+2z_j+z_{j+1}+\cdots+z_{n-i})+\text{higher degree terms}
\end{split}
\end{equation*}

By Nakayama's lemma, we only need to show that the vectors $z_1+\cdots+z_{j-1}+2z_j+z_{j+1}+\cdots+z_{n-i}$, $1\leq j\leq n-i$ span the whole vector space $\cc z_1\oplus \cc z_2\oplus\cdots \oplus \cc z_{n-j}$. By adding all such vectors together, we see $z_1+z_2+\cdots+z_{n-i}$ is contained in their span. Thus 
$$z_j=(z_1+\cdots+z_{j-1}+2z_j+z_{j+1}+\cdots+z_{n-i})-(z_1+z_2+\cdots+z_{n-i})$$
 is in the span. 
\end{proof}
Now, Proposition  \ref{num} follows from Theorem \ref{1st} and Proposition \ref{2nd}. 

\section{The constants $\beta_{\mu,\nu}$}\label{four}
Instead of working with the constants $\beta_{\mu, \nu}$, we define $\alpha_{\mu, \nu}$ to be the number of complex solutions to the system of equations
\begin{equation}\label{system}
\begin{cases}
z_1^\nu=z_2^\nu=\cdots=z_\mu^\nu=1\\
z_1+z_2+\cdots+z_\mu=0. 
\end{cases}
\end{equation}
Then clearly $\beta_{\mu, \nu}=\frac{1}{\nu}\cdot\alpha_{\mu+1, \nu}$. The advantage of working with $\alpha_{\mu, \nu}$ is that their defining equations have better symmetry. 

We would like to answer the following question.
\begin{question}
Give a formula for $\alpha_{\mu, \nu}$ in terms of $\mu$ and the prime factorization of $\nu$. 
\end{question}
This is definitely a very hard question. The work of Lam and Leung gives a necessary and sufficient condition of $\alpha_{\mu, \nu}\neq 0$. 
\begin{theorem}{\cite{LL}}
Suppose $\nu=p_1^{a_1}\cdots p_l^{a_l}$ is the prime factorization. Then $\alpha_{\mu, \nu}\neq 0$ if and only if $\mu\in \zz_{\geq 0}\cdot p_1+\cdots+\zz_{\geq 0}\cdot p_l$. 
\end{theorem}
When $\nu=p^r$ has only one prime factor, we can give a formula of $\alpha_{\mu, \nu}$. In this case, suppose $(z_1, \ldots, z_\mu)$ is a solution to (\ref{system}). Then the collection $\{z_1, \ldots, z_\mu\}$ can be divided into groups of $p$ elements such that each group is a rotation of $1, e^{2\pi i/p}, \ldots, e^{2(p-1)\pi i/p}$. Therefore, if $p$ does not divide $\mu$, then $\alpha_{\mu, \nu}=0$. If $p$ divides $\mu$, then 
\begin{equation}\label{ab}
\alpha_{\mu, \nu}=\sum \frac{\mu!}{\left((s_1)!(s_2)!\cdots(s_{k})!\right)^p}
\end{equation}
where $k=\nu/p$, and the sum is over all $s_1, \ldots, s_k\in \zz_{\geq 0}$ such that $s_1+\cdots+s_k=\mu/p$. Since $\beta_{\mu, \nu}=\frac{1}{\nu}\cdot\alpha_{\mu+1, \nu}$, we can translate (\ref{ab}) into a statement about $\beta_{\mu, \nu}$.
\begin{prop}\label{com}
Suppose $\nu=p^r$, where $p$ is a prime number and $r$ is a positive integer. Then $\beta_{\mu, \nu}=0$ when $p$ does not divide $\mu+1$, and when $p$ divides $\mu+1$
\begin{equation}
\beta_{\mu, \nu}=\frac{1}{\nu}\sum \frac{(\mu+1)!}{\left((s_1)!(s_2)!\cdots(s_{k})!\right)^p}
\end{equation}
where $k=\nu/p$, and the sum is over all $s_1, \ldots, s_k\in \zz_{\geq 0}$ such that $s_1+\cdots+s_k=\frac{\mu+1}{p}$. 
\end{prop}

Suppose $\nu=p^rq^s$ has two distinct prime factors, and suppose $(z_1, \ldots, z_\mu)$ is a solution to (\ref{system}). Then by \cite[Corollary 3.4]{LL}, the collection $\{z_1, \ldots, z_\mu\}$ can be divided into groups of $p$ or $q$ elements such that each group is a rotation of $1, e^{2\pi i/p}, \ldots, e^{2(p-1)\pi i/p}$ or a rotation of $1, e^{2\pi/q}, \ldots, e^{2(q-1)\pi i/q}$ respectively. However, this decomposition is not unique, and this is the main difficulty to find a formula for $\alpha_{\mu, \nu}$ in this case. Now, this is already a problem beyond our capability. 

When $\nu$ has at least three distinct prime factors, the statement of \cite[Corollary 3.4]{LL} is not true any more. Therefore, the question becomes much harder and deeper.

\end{document}